\documentclass{siamart190516}

\usepackage{algorithm}
\usepackage{algpseudocode}
\usepackage{lipsum} 
\usepackage{amsfonts}

\usepackage{accents}
\usepackage{hhline}
\usepackage{calc}
\usepackage{subfigure}
\usepackage{stfloats}
\usepackage{midfloat}
\usepackage{mathrsfs}
\usepackage{mdwtab}
\usepackage{mathtools}
\usepackage{multirow}

\newcommand{\mat}[1]{{\bf{#1}}}
\newcommand{\mathat}[1]{{\bf{\hat{#1}}}}

\newcommand{\matcal}[1]{{\boldsymbol{{\mathcal{#1}}}}}
\newcommand{\mattilde}[1]{\mat{\widetilde{#1}}}
\newsiamremark{remark}{Remark}
\newcommand{\mats}[1]{{\boldsymbol{#1}}}
\newcommand{\mathats}[1]{{\boldsymbol{\hat{#1}}}}

\newenvironment{Smallmatrix}[1]
  {\arraycolsep=3pt\footnotesize
   \array{#1}}
  {\endarray}

\begin{document}
\headers{Obreshkov Order in DAE}{Emad Gad}
\title{Characterizing Order of Convergence in the Obreshkov Method in Differential-Algebraic Equations}
\author{Emad Gad \thanks{School of Electrical Engineering and Computer Science, Ottawa, ON, Canada K1N 5N6 (\email{egad@uottawa.ca} })
\funding{This work was supported by the Natural Sciences and Engineering Research Council (NSERC) of Canada.}}
\maketitle

\begin{abstract}
The Obreshkov method is a single-step multi-derivative method used in the numerical solution of differential equations and has been used in recent years in efficient circuit simulation. It has been shown that it can be made of arbitrary high local order of convergence while maintaining unconditional numerical stability. Nevertheless, the theoretical basis for the high order of convergence has been known only for the special case where the underlying system of differential equations is of the ordinary type, i.e., for ordinary differential equations (ODE). On the other hand, theoretical analysis of the order of convergence for the more general case of a system consisting of differential and algebraic equations (DAE) is still lacking in the literature. 

This paper presents the theoretical characterization for the local order of convergence of the Obreshkov method when used in the numerical solution of a system of DAE. The contribution presented in this paper demonstrates that, in DAE, the local order of convergence is a function of the differentiation index of the system and, under certain conditions, becomes lower than the order obtained in ODE.  

\end{abstract}


\section{Introduction}
\label{sec:introduction}
Numerical solution of differential equations (DEs) is one of the fundamental tools used in all walks of applied and computational sciences. The main goal in any method used to solve a system of DEs of the form 
\begin{equation}
\label{eq:DAE-nonlinear}
\mat{f}\left(\mat{x}(t),\frac{\textnormal{d}\mat{x}(t)}{\textnormal{d}t},t\right) = \mat{0},
\end{equation}
(with $\mat{x}(t)\in\mathbb{R}^N, \mat{f}: \mathbb{R}^N \times
\mathbb{R}^N \times \mathbb{R} \rightarrow \mathbb{R}^N$) is to
approximate $\mat{x}(t)$ at  time points, $t_n,n=0,1,2\cdots$. The process of computing approximations to $\mat{x}(t_i)$ is typically known as time marching, and starts from an initial point, $t_0$,where $\mat{x}(t_0)$ is known. In many situations, it is possible that the derivatives $\frac{\textnormal{d}\mat{x}(t)}{\textnormal{d}t}$ can be expressed explicitly in terms of $\mat{x}(t)$ and $t$, in which case the system of DEs is said to be in the ordinary differential equations (ODE) form. If that is not possible, the DEs are referred to as a system of differential-algebraic equations (DAE).

Research activities targeting developing numerical methods for solving DEs span over several decades and are still going strong. Section \ref{sec:general-background} provides a more detailed account of those activities. One particular method that has been applied recently in the area of circuit simulation \cite{gad09:_A_L_stable_high_order,GAD-2012-6106731,Gad_Mina_2013_6407945,Gad_2013_Mina_6359804,pimentel11:_high_order_stabl_and_stabl,Gad_Yaoyao_6935006} is based on the Obreshkov formula \cite{obreshkov42:_sur_les_quadr_mecan}. The Obreshkov-based numerical method was shown to provide key features foremost among them is the fact that the method can be implemented with arbitrary high-order without losing the $A$-stability or the $L$-stability properties. In addition, the parameters of the method (coefficients of the Obreshkov formula) are given in pre-determined analytical form. These features made the application of this method advantageous in circuit simulation.

Nevertheless, the exitsting literature do not provide a characterization of the Obreshkov order of convergence if the underlying DEs is in the form of DAE. Indeed, the order of convergence in the Obreshkov-based method is only known if the DEs take the form of ODE. The lack of rigorous analysis for Obreshkov order of convergence in DAE is viewed as a gap in the literature and is being addressed in this paper. The analysis and obtained results presented in this paper show that the order of the Obreshkov-based method is not determined only by the parameters of the integration formula, as in the case of ODE, but also by the differentiation index of the DAE. More particularly, it is shown that the differentiation index of the has the potential DAE to lower the order of the Obreshkov method below the order nominally afforded by its parameters.
\subsection{Organization of the Paper}
\label{sec:organisation-paper}
Section \ref{sec:general-background} provides the background to the topic addressed in this paper. It also provides a substantial overview of the various methods proposed in the literature to numerically solve DEs, and positions the Obreshkov method within this body of work highlighting the scope of the new contribution. Section \ref{sec:obreshk-based-numer} describes the recent progress on using the Obreshkov formula in circuit simulation underscoring its high-order for the special case of ODE. Section \ref{sec:order-char-dae} develops the mathematical analysis needed to characterize the Obreshkov order in DAE. Section \ref{sec:exper-valid} provides experimental validation of the theoretical results. Finally, Section \ref{sec:conclusion} draws important conclusion from the presented work.

\section{General Background and Scope of the Paper}
\label{sec:general-background}
This section describes the main landscape of the methods used in the numerical solution of DEs. Section \ref{sec:numer-meth-solv} dwells on the major criteria through which those methods can be viewed and classified. Section \ref{sec:disc-ideal-meth} underscores the ideal characteristics that is typically sought in a general-purpose method and reviews the literature describing how the various methods derived score on those characteristics. The review of the Obreshkov method and its recent utilization in circuit simulation is presented in Section \ref{sec:backgr-obreshk-meth}. Finally, Section \ref{sec:scope-this-work} presents the scope of the contribution of the paper.
\subsection{Notation}
\label{sec:notation}
$\mat{x}(t_n)$ will be taken as the \textit{exact} value of the solution to the DEs at $t=t_n$. On the other hand, the approximation generated by any given solution method at $t=t_n$ will be denoted $\mathat{x}_n$. In a similar manner, the exact $i^{\mathrm{th}}$ order derivative of $\mat{x}(t)$ at $t=t_n$ will be denoted by $\mat{x}^{(i)}\left(t_n\right)$, whereas its approximation is denoted by $\mathat{x}_n^{(i)}$, $i=0,1,2,\cdots$, with the convention that $\mat{x}^{(0)}\left(t_n\right)\equiv \mat{x}(t_n)$ and $\mathat{x}^{(0)}\equiv\mathat{x}_n$. $h$ in the context of this paper will refer to the latest step size, $h=t_{n}-t_{n-1}$. $\mathcal{O}\left(h^q\right)$ will mark series terms that are proportional to $h^u$, with $u \geq q$ and $\matcal{O}_r\left(h^q\right)$ is used to extend this notation to a vector of size $r$. Finally, $\mat{I}_H$ will be used to denote $H \times H$ identity matrix.

\subsection{Overview of Numerical Methods for Solving DEs}
\label{sec:numer-meth-solv}
Several themes pervade the literature of methods that numerically solve DEs as an Initial Value Problem (IVP). This section reviews those themes treating them as distinct classes noting that a particular method can lie at the intersection of two or more classes.

\subsubsection{Classification Based on the Construction of the Method}
\label{sec:class-meth-based-1}
The methods which have been proposed, or used in software packages, can be classified based on the approach they use to advance from point $t_{n-1}$ to $t_{n}$. Those methods are often grouped under the following three categories.
\begin{enumerate}
\item \textbf{Linear Multi-Step (LMS).} These methods rely on the approximations generated at the past $r$ points, that is
  $\mathat{x}_{n-1},\mathat{x}_{n-2},\cdots,\mathat{x}_{n-r} $ and the first-order derivatives at those points to obtain an approximation to $\mat{x}\left(t_{n}\right)$. LMS methods are usually represented by the following formula,
\begin{equation}
  \label{eq:LMS formulae}
  \sum_{l=0}^{r}\mu_l\mathat{x}_{n-l} = h\sum_{l=0}^{r}\zeta_l \mathat{x}^{(1)}_{n-l}
\end{equation}
where the coefficients $\mu_i,\zeta_i$ are specific to each integration method and are chosen to make the first $q$ Taylor series terms in the operator $\matcal{L}(\mat{x}(t)):=\sum_{l=0}^{r}\mu_l\mat{x}(t_{n-l}) - h\sum_{l=0}^{r}\zeta_l \mat{x}^{(1)}(t_{n-l})$ vanish \cite{gear71:_numer}. Notable example among these methods, are the backward Euler (BE), trapezoidal rule (TR) and the backward differentiation formulae (BDF).
\item \textbf{Single-Step Multi-Stage (SSMS).}
  Those methods use the approximation at a single past time point, $t_{n-1}$, along with approximations to off-step points (known as stages), at $t_i=t_n+c_ih_n$, where $c_i< 1$, $i=1,\cdots,s$. The Runge-Kutta (RK) methods \cite{butcher03:BOOK_numer_method_ordin_differ_equat} are the best known example in this class of methods. An RK-method is typically represented by the Butcher tableau $\begin{Smallmatrix}{c|c}\mat{c} & \mat{A}\\  \hline& \mat{b} \end{Smallmatrix}$, where $\mat{A}\in\mathbb{R}^{s\times s},\mat{b},\mat{c}\in\mathbb{R}^s$ with $s$ the number of stages. The formulation of the RK for a system of ordinary differential equations $\mat{x}^{(1)}=\mat{f}\left(t,\mat{x}\right)$ takes the following form
\begin{eqnarray*}
  \label{eq:RK-formulation}
  \mathat{g}_i &= & \mat{f}\left(t_{n-1}+c_ih,h\Sigma_{j=1}^{s}{a_{ij}}{\mathat{g}_i}\right) \quad i=1,\cdots,s\\
  \mathat{x}_{n} &= & \mathat{x}_{n-1} +h\sum_{i=1}^{s}b_i\mathat{g}_i
\end{eqnarray*}
where $\mathat{g}_i$ are called the \textit{stage value} that approximate $ \mat{x}(t_{n-1}+c_ih)$ and $a_{ij}, b_i,c_i$ are the components of $\mat{A},\mat{b}$ and $\mat{c}$, respectively.

\item \textbf{General Linear Methods (GLM).} The GLM are typically viewed as a hybrid between the LMS and SSMS methods, since it uses the past $r$ points along with $s$ stages to advance
to the next step. The construction of a GLM method is usually represented by the block matrix 
$
\left[
  \begin{array}[c]{c c}
    \mat{A} & \mat{U}\\
    \mat{B} &\mat{V}
  \end{array}
\right]
$
where $\mat{A}\in\mathbb{R}^{s\times s},\mat{B}\in{\mathbb{R}^{r\times s}},\mat{U}\in\mathbb{R}^{s\times r}$, $\mat{V}\in\mathbb{R}^{r\times r}$. The time stepping to $\mathat{x}_n$ is presented as
\begin{eqnarray}
  \label{eq:GLM}
  \mathat{g}_i &=& a_{ij}h\mat{f}\left(t_{n-1}+c_ih,h\Sigma_{j=1}^{s}{a_{ij}}{\mathat{g}_i}\right)  + \Sigma_{j=1}^ru_{ij}\mathat{x}_{n-j},\quad i=1,\cdots,s\nonumber\\
  \mathat{x}_{n-i+1} &=& \Sigma_{j=1}^sb_{ij}h\mathat{g}_j+\Sigma_{j=1}^rv_{ij}\mathat{x}_{n-j},\quad i=1,\cdots,r\nonumber
\end{eqnarray}
where $u_{ij}$, $v_{ij}$, $a_{ij}$, $b_{ij}$ are the components of $\mat{U}$, $\mat{V}$, $\mat{A}$ and $\mat{B}$, respectively.
\end{enumerate}

\subsubsection{Classification Based on Stability Criteria}
\label{sec:class-meth-based}
The stability properties of a given integration method is usually studied through characterizing its behavior in approximating the solution to the scalar test problem, defined by,
\begin{equation}
  \label{eq:Scalar Test problem}
  \frac{\mathrm{d} x(t)}{\mathrm{d} t} = \lambda x(t)
\end{equation}

A method is said to be $A$-stable if the successive approximations
to the scalar test problem satisfy $|\hat{x}_j|< |\hat{x}_i|$, $(j>i)$ for all
values of $\mathbb{C}^-$, $L$-stable if $\lambda\in\mathbb{C}^-\cup
\{\infty\}$, and $A(\alpha)$-stable if
$\lambda\in\left\lbrace\lambda;|\arg(-\lambda)|<\alpha,\lambda\neq
0\right\rbrace$. Note here that $L$-stability implies necessarily $A$-stability. 
\subsubsection{Classification Based on the Order of the Method}
\label{sec:class-based-order}
A third way to characterize a given method is to derive the relation between the approximation $\mathat{x}_n$ and the exact value $\mat{x}(t_n)$ in terms of the step size $h=t_{n}-t_{n-1}$, assuming that approximations generated at, or prior to, $t_{n-1}$ are exact. Such relation defines the ``local order'' of the method and is made more precise through the following definition.
\begin{definition}[Local Order of Convergence]
  \label{def:order-of-convergence}
  An $r$-step  method used to approximate the solution of a general differential equations is said to be of local order $q$ if the approximate solution, $\hat{x}_n$ is related to the exact solution, $x(t_n)$, through,
  \begin{equation}
    \label{eq:1}
    \mathat{x}_n = \mathat{x}(t_n) + C h^{q+1}\left.\frac{\textnormal{d}^{q+1}}{\textnormal{d}t^{q+1}}\mat{x}(t)\right|_{t=t_{n-1}}+\matcal{O}_N(h^{q+2})
  \end{equation}
  assuming that $\mathat{x}_{n-p} = \mat{x}\left(t_{n-p}\right)$, for $p=1,\cdots,r$.
\end{definition}
\subsubsection{Classification Based on Explicit vs. Implicit methods}
\label{sec:class-based-expl}
Within the above classes, a method can be further classified as explicit vs. implicit. For example, LMS methods are implicit if its integration formula (\ref{eq:LMS formulae}) includes both the (unknown) $\mathat{x}_n$ and its derivative $\mathat{x}^{(1)}_n$, i.e., if $\zeta_0\neq 0$. Otherwise, it is called explicit. SSMS or GLM methods are considered explicit if the stage $\mathat{g}_n$ depends strictly on the stages $\mathat{g}_m$ for  $m < n$. This would be the case if the matrix $\mat{A}$ is strictly lower triangular, i.e., $a_{ij}=0$ for $j\ge i$. Otherwise it is considered implicit. Explicit methods have the advantage that computing the unknown is only done through linear combination of known values, with complexity proportional to $N$, $\mathcal{O}(N)$, while implicit methods require a linear system solution, which scales in the general case as $\mathcal{O}\left((sN)^{\alpha}\right)$, $s$ being the number of stages, and $1 < \alpha \le 3 $.  

\subsection{Literature Overview}
\label{sec:disc-ideal-meth}

The common thrust in the methods proposed under the above classes has been a quest for an unconditionally stable method with high order and a computational complexity that scales modestly with the problem size $N$. In the domain of circuit simulation, for instance, the requirement of unconditional stability is found in the general class of $A$-stable methods. Therefore, it becomes imperative that a method constructed using one the above three approaches mentioned in Section \ref{sec:class-meth-based-1} satisfy the $A$-stability condition indicating the local order of approximation, with higher order being more desirable than lower order.

Dahlquist, in his landmark paper \cite{Dahl63}, showed that $A$-stable methods constructed through the LMS approach cannot have order higher than 2. It was also shown that explicit methods constructed through the LMS approach cannot be $A$-stable. This result, rightly termed the Dahlquist barriers \cite{hiarer08:_solvin_ordin_differ_equat_i_nonst_probl}, put to rest the quest for deriving LMS $A$-stable methods of any order higher than 2 $(q\le 2)$. As a result of those barriers, achieving order higher than 2 using the LMS methodology can only be carried out through relaxing the $A$-stability requirement and adopting $A(\alpha)$-stability instead, with the (BDF) method being the best-known example in these methods.

Unlike LMS-constructed methods, SSMS methods such as the RK do not have a theoretical barrier that precludes the existence of arbitrary high-order $A$-stable methods. For example, it is known that the highest-order for $A$-stable RK method is only restricted by the number of stages $s$, and the identity that the order $q$ satisfies $q \le 2s$ was first proposed by the Daniel-Moore conjecture \cite{daniel70:_comput} before being proved by the theory of order stars \cite{iserles91:_order_stars,wanner78:_order_stars_stabil_theor}. Nevertheless, the difficulty in finding high-order RK methods is not due to the existence of a theoretical barrier but rather in constructing the appropriate Butcher tableau. For example, to
derive the Butcher tableau corresponding to a method of order $10$, one would need to solve a system of $1205$ equations with extremely complex nonlinear behavior, where only a limited set of the solutions may be feasible \cite{butcher03:BOOK_numer_method_ordin_differ_equat}. In addition, actual implementation of RK methods requires the solution of $sN$ strongly coupled nonlinear equations to obtain the stage values, $\mathat{g}_i$, 
\cite{burrage95:_paral_sequen_method_ordin_differ_equat}.

Similar to SSMS, GLM methods also do not have theoretical ``barriers'' on the order for the $A$-stability. Rather, its stability is characterized through a two-dimensional polynomial $\mat{\Phi}(w,z):=\det\left(w\mat{I}-\mat{M}(z)\right)$, with $\mat{M}(z) = \mat{V}+z\mat{B}\left(\mat{I}-z\mat{A}\right)^{-1}\mat{U}$ \cite{butcher06:ACTA_NUMERICA__gener_linear_method}. In fact, $A$-stability is guaranteed in the GLM so long as  all the $w$-roots of the polynomial $\mat{\Phi}(w,z)=0$ lie inside the unit disk of the complex plan whenever $\Re{\{z\}}<0$, with only simple roots allowed at the unit circle. The barrier on the highest possible order $A$-stable GLM is given by the following theorem.
\begin{theorem}(\cite{wanner78:_order_stars_stabil_theor})
  Let $\mat{\Phi(w,z)}$ be a two-dimensional polynomial corresponding to an $A$-stable GLM method of order $q$. Then $q\le 2s$. Furthermore, methods attaining the highest possible order, i.e. for $p=2s$, will have an error constant satisfying the inequality $C\ge (-1)^{s}\frac{s!s!} {(2s)!(2s+1)!}$.
  
\end{theorem}
 Examples GLM are the diagonally implicit multi stage methods (DIMSIM) methods
\cite{butcher93:DIMISIM_diagon_implic_multi_stage_integ_method}
and multistep collocation techniques
\cite{lie89:_super_multis_colloc}

To the best of the author's knowledge, the construction of $A$-stable GLM method is still an open question. However, as the above theorem states, the only limit to the order of any such method, when it is found, will be determined by the number of stages $s$ used in the construction.

\subsection{Background on the Obreshkov Method}
\label{sec:backgr-obreshk-meth}

The Obreshkov method, although dating back to 1942 \cite{obreshkov42:_sur_les_quadr_mecan}, has not received serious attention during the above developments. In fact, this method can be considered as a fourth class of methods as far as the its construction approach is considered, since it relies on higher-order derivatives. The method is constructed by forcing two successive approximations, $\mathat{x}_{n-1},\mathat{x}_{n}\in\mathbb{R}^N$ at $t =t_{n-1}, t_{n}$, to satisfy the Obreshkov formula \cite{obreshkov42:_sur_les_quadr_mecan,ehle68:_high_order_method_numer_solut_d}, which is given by,
\begin{equation}
  \label{eq:def-of-the-Obreshkov-formula}
  \sum_{i=0}^m (-1)^i\alpha_{i,l,m} h^i \mathat{x}^{(i)}_{n} = \sum_{i=0}^l \alpha_{i,m,l} h^i \mathat{x}^{(i)}_{n-1}
\end{equation}
where
\begin{equation}
  \label{eq:The-aplha-factors}
  \alpha_{i,l,m} = \frac{(m+l-i)!}{(m+l)!}\left(
    \begin{array}[c]{c}
      m\\i
    \end{array}
  \right),
\end{equation}
wehre $l$ and $m$ are integers that define the parameters of the formula. It should be obvious that the Obreshkov method is a single-step method in the sense that it uses the immediate past approximation (for $t=t_{n-1}$) to advance to the time point at $t=t_n$. However, in contrast to LMS methods, it requires the higher-order derivatives at the past and current time point. Furthermore, the method is implicit since its formula includes the derivatives at the unknown point $t_n$.

In 1968, Ehle \cite{ehle68:_high_order_method_numer_solut_d} showed that Obreshkov formula, when implemented in a single-step implicit mode with specific parameters, $m-2 \le l\le m$, can be used to derive $A$-stable and $L$-stable methods. One salient advantage to the Obreshkov-based approach, compared to the RK or the GLM, was the fact the coefficients $\alpha_{i,l,m}$  defining the method are pre-determined and are given by the expression in \eqref{eq:The-aplha-factors} for arbitrary high-order. Despite its potential advantages, Ehle remarked that this integration formula appears to be “largely of theoretical interest” due to having to work with implicit high-order derivatives \cite{ehle68:_high_order_method_numer_solut_d}. Consequently, interest in working toward a practical implementation of the Obreshkov formula in general nonlinear systems waned. In fact, Gear, in his classical work on BDF methods \cite{gear71:_numer}, notes that due to such a difficulty, “the application of these methods is a major computing task for large systems and is not generally practical.” Such an attitude toward using Obreshkov-like formulas persisted over three decades with only one exception; when N{\o}rsett attempted to modify them in an effort to reduce the computational difficulties in solving linear ordinary DEs. The modified method, however, lost the $A$-stability property for orders higher than five. In addition, its generalization to nonlinear DAE-based systems was not obvious \cite{norsett74:_one_step_method_hermit_type}.

In a more recent development, research in the area of circuit simulation renewed the interest in an Obreshkov-based approach to the numerical solution of DEs. In \cite{gad09:_A_L_stable_high_order}, a graphical methodology, called ``rooted trees'', was used to represent the nonlinear functions used in nonlinear circuit device models and proved successful in computing the high-order derivatives in an efficient manner, thereby removing the main obstacle that impeded using the Obreshkov method in industrial-scale problems. Following this work, the Obreshkov method was tailored to more specialized circuits in \cite{Gad_Mina_2013_6407945,Gad_2013_Mina_6359804,Gad_Yaoyao_6935006,Gad_Farhan_2017_7480392}, with increasing success compared to the conventional methods based on LMS that are used in circuit simulations. Further characterization of the Obreshkov method was also presented in \cite{GAD-2012-6106731} along with more efficient implementation techniques.

\subsection{Scope of this Work}
\label{sec:scope-this-work}
The work presented in this paper is motivated by the fact that the existing literature provides a characterization for the order of the Obreshkov method \textit{only} when it is being applied to a system of ordinary differential equations (ODE). On the other hand, the treatment of the more general case of DAE is still lacking in literature. It is a well-known fact that DAE have fundamentally different and richer structural characteristics from those of ODE, and that those characteristics can cause a particular method to behave differently when applied to DAE (compared with its application to ODE). For example, in the class of SSMS-RK methods, the order in solving DAE deviates from the order of solving ODE \cite{ascher98:BOOK_comput_method_ordin_differ_equat,kunkel06:_differ_algeb_equat_analy_and_numer_solut, maffezzoni07:_time_runge_kutta}, leading to loss of order. The core contribution in this work expands the existing theory, which describes the order of Obreshkov method in ODE, to cover the general case of DAE. It demonstrates that the order of the Obreshkov method, when used to solve DAE, is a function of the system differentiation index \cite{kunkel06:_differ_algeb_equat_analy_and_numer_solut}. It is shown in this work that DAE with high differentiation index do in fact lower the order of the Obreshkov method, which is a results that, to the best of the author's knowledge, is novel.

\section{Implementation and Order of the Obreshkov Method}
\label{sec:obreshk-based-numer}
Characterization of the order achieved by a given numerical method can be carried out by considering a linear version of the DEs in (\ref{eq:DAE-nonlinear}). For the case of an ODE, such DEs are represented by
\begin{equation}
  \label{eq:The-ODE}
  \frac{\mathrm{d}\mat{x}(t)}{\mathrm{d}t} = \mat{A}\mat{x}(t) +\mat{u}(t)
\end{equation}
where $\mat{A}$ is a matrix in $\mathbb{R}^{N\times N}$. On the other hand, the DAE formulation for linear DEs can be put in the form
\begin{equation}
  \label{eq:The-DAE}
  \mat{C}\frac{\mathrm{d}\mat{x}(t)}{\mathrm{d}t}+\mat{G}\mat{x}(t) =\mat{b}(t) 
\end{equation}
where $\mat{C}$ and $\mat{G}$ are matrices in $\mathbb{R}^{N\times N}$ with $\mat{C}$ being generally singular matrix. In case that the matrix $\mat{C}$ is nonsingular, then (\ref{eq:The-DAE}) can be put in the form of the ODE (\ref{eq:The-ODE}) using $\mat{A}=-\mat{C}^{-1}\mat{G}$ and $\mat{u}(t)=\mat{C}^{-1}\mat{b}(t)$. 

In the circuit simulation framework, the matrix $\mat{C}$ carries the so-called ``stamps'' of memory elements such as inductors and capacitors, while $\mat{G}$ carries the stamps of the memoryless elements (including resistors, independent/dependent voltage/current sources), while the vector $\mat{b}(t)\in\mathbb{R}^N$ groups the contributions of the independent voltage/current sources in the circuit. The standard approach used to construct those matrices is based on the modified nodal analysis (MNA) formulation \cite{mna_1975,najm2010circuit}.

\subsection{Utilizing the Obreshkov Formula in Circuit Simulation}
\label{sec:util-obreshk-form}
The application of the Obreshkov  formula in the context of circuit simulation was proposed recently in \cite{gad09:_A_L_stable_high_order,GAD-2012-6106731,Gad_Mina_2013_6407945,Gad_2013_Mina_6359804}. In this approach to using the Obreshkov formula \eqref{eq:def-of-the-Obreshkov-formula} on (\ref{eq:The-DAE}), the time stepping from $t=t_{n-1}$ to $t_n$ is done by solving the following system
\begin{equation}
  \label{eq:The Augmented MNA for the Obreshov formula}
  \left(\mattilde{C} + \mattilde{G}\right)\mathats{\xi}_{n} = \mattilde{b}_{n}
\end{equation}
The matrices $\mattilde{C}$ and $\mattilde{G}\in\mathbb{R}^{(k+1)N\times(k+1)N}$ in (\ref{eq:The Augmented MNA for the Obreshov formula}) are obtained from $\mat{C}$ and $\mat{G}$ as follows
\begin{equation}
  \label{eq:The Augmented C matrix}
  \mattilde{C} =
  \frac{1}{h}
  \left[
    \begin{array}[c]{c c c c c }
      \mat{0} & \mat{C} & \cdots  &        & \mat{0}\\
      \mat{0} & \mat{0} & \mat{C} & \cdots & \mat{0}\\
      \vdots  & \vdots  & \ddots  &        & \mat{\vdots}\\
      \mat{0} & \mat{0} &  \cdots & \mat{0}& \mat{C}\\
      \mat{0} & \mat{0} &         & \cdots & \mat{0}
    \end{array}
  \right]
\end{equation}
\begin{equation}
  \label{eq:The Augmented G matrix}
  \mattilde{G} =
  \left[
    \begin{array}[c]{c c c c c }
      \mat{G}     & \mat{0}     & \cdots  &        & \mat{0}\\
      \mat{0}     & \mat{G}     & \mat{0} & \cdots & \mat{0}\\
      \vdots      & \vdots      & \ddots  &        & \mat{\vdots}\\
      \mat{0}     & \mat{0}     & \cdots  & \mat{G}& \mat{0}\\
      \alpha_{0,l,m}\mat{I}_N & -\alpha_{1,l,m}\mat{I}_N &         & \cdots & (-1)^{k}\alpha_{m,l,m}\mat{I}_N
    \end{array}
  \right]
\end{equation}
and the vectors $\mattilde{b}_{n}$ and $ \mathats{\xi}_{n} $ are given by
      \begin{equation}
        \label{eq:b-tilde}
        \mattilde{b}_{n} =
      \left[
        \begin{array}[c]{c c c c }
          \left({\mat{b}^{(0)}(t_{n})}\right)^{\top} &          \cdots & \left({h^{m-1}\mat{b}^{(m-1)}(t_{n})}\right)^{\top}  & \left(\sum_{i=0}^{l}\alpha_{i,m,l} h^i {\mathat{x}^{(i)}_{n-1}}\right)^{\top}
        \end{array}
        \right]^{\top}
      \end{equation}

\begin{equation}
  \label{eq:The augmented ssystem the vector xi}
  \mathats{\xi}_{n} =
  {\left[
    \begin{array}[c]{c c c c}
      {\mathat{x}^{(0)}_{n}}^{\top} & h{\mathat{x}^{(1)}_{n}}^{\top}& \cdots & h^m{\mathat{x}^{(m)}_{n}}^{\top}
    \end{array}
  \right]}^{\top}
\end{equation}
where $\mat{b}^{(i)}(t_n)=\frac{\mathrm{d}^i\mat{b}(t)}{\mathrm{d}t^i}$ evaluated at $t=t_n$. The above implementation of the Obreshkov method in circuit simulation was shown to be computationally advantageous as it preserves the inherent sparsity of the matrices $\mat{G}$ and $\mat{C}$ that arise from the circuit formulation and enables using block forms of LU factorization such as the KLU \cite{DBLP:journals/toms/DavisN10} which scales almost linearly with the circuit size and the order $m$ \cite{Gad_Yaoyao_6935006}. The initial point of the time marching is done by computing $\mat{x}^{(i)}(0)$, $i=0,\cdots,l$, and the time marching proceeds from $t=t_{n-1}$ to $t_n$ with appropriately sized time step  as shown in \cite{gad09:_A_L_stable_high_order}.

    \subsection{Order of the Obreshkov Formula with ODE}
    \label{sec:order-obreshk-form}
    The first result listed here is concerned with the order of the Obreshkov method when applied to a system of ODE (\ref{eq:The-ODE}).
    \begin{lemma}
      \label{lemma:order-obreshk-form--ODE}
      Assume that the $\mat{C}$ matrix is non-singular and, hence, the DAE system (\ref{eq:The-DAE}) can be put in the form of the ODE (\ref{eq:The-ODE}). Also assume that
      \begin{equation}
        \label{eq: lemma-assumption}
        \mathat{x}_{n-1}^{(i)} = \left.\frac{\mathrm{d}^{i}\mat{x}(t)}{\mathrm{d}t^{i}}\right|_{t=t_{n-1}},\quad i=0,\cdots,m
      \end{equation}
      It then follows that
      \begin{equation}
        \label{eq:Order-ODE-lemma}
        h^i\mathat{x}_n^{(i)} = h^i\mat{x}^{(i)}\left(t_n\right) + \matcal{O}_N\left(h^{l+m+i+1}\right),\quad i=0,\cdots,m
      \end{equation}
    \end{lemma}
The proof of Lemma \ref{lemma:order-obreshk-form--ODE} will be given as part of the proof of the more general result pertinent to the case of DAE, in Section \ref{sec:diff-error-comp}.
    \begin{remark}
      The local order of convergence for the actual waveform ($i=0$) in the Obreshkov method when applied to a system of ODE is determined solely based on the parameters $l$ and $m$ used in the Obreshkov formula.
    \end{remark}
    \begin{remark}
      The $i^{\mathrm{th}}$ order derivative is approximated to the same degree as the lower-order derivatives. This fact follows from \eqref{eq:Order-ODE-lemma}, or alternatively from $  \mathat{x}_n^{(i)} - \mat{x}^{(i)}\left(t_n\right)$  is given by $\matcal{O}_N\left(h^{l+m+1}\right)$, independent from $i$. 
    \end{remark}

    \section{Obreshkov Order Characterization in DAE}
    \label{sec:order-char-dae}
    
\subsection{Preliminaries}
\label{sec:preliminaries}
The development of the analysis is predicated on the following prelimiaries. 
\begin{definition}[Differentiation Index]
  The minimum number of times that all or part of the DAE (\ref{eq:The-DAE}) must be differentiated with respect to $t$ in order to explicitly express the derivative $\frac{\mathrm{d}\mat{x}(t)}{\mathrm{d}t}$ in terms of $\mat{x}(t)$ and $t$ is defined as the \textnormal{differentiation index} of the DAE. 
\end{definition}
\begin{definition}[Nilpotent matrix]
  \label{def:preliminaries-1-Matrix Nilpotency Index}
  A nilpotent square matrix $\mat{N}$ is defined such that $\mat{N}^\nu=\mat{0}$ and $\mat{N}^{\nu-1}\neq \mat{0}$ for some positive integer $\nu$. $\nu$ is referred to as the \textnormal{nilpotency index} of $\mat{N}$.
\end{definition}
\begin{definition}[Solvability of the DAE \cite{ascher98:BOOK_comput_method_ordin_differ_equat}]
  \label{def:preliminaries-1-Solvability}
  Let $\mathcal{I}$ be an open interval of $\mathbb{R}$, $\Omega $ a connected open subset of $\mathbb{R}^{2N+1}$ and $\mat{f}$ of (\ref{eq:DAE-nonlinear}) differentiable from $\Omega $ to $\mathbb{R}^N$. Then the DAE (\ref{eq:DAE-nonlinear}) is \textnormal{solvable} on $\mathcal{I}$ in $\Omega$ if there is an $r$-dimensional family of solutions $\phi(t,c)$ defined on a connected open set $\mathcal{I}\times\bar{\Omega},\bar{\Omega}\subset\mathbb{R}^r$, such that:
  \begin{enumerate}
  \item $\phi(t,c)$ is defined all of $\mathcal{I}$ for each $c\subset\bar{\Omega}$
  \item $(t,\phi(t,c),\phi^{(1)}(t,c))\in\Omega$ for $(t,c)\in\mathcal{I}\times\bar{\Omega}$
  \item if $\psi(t,c)$ is any other solution with $(t,\psi(t,c),\psi^{(1)}(t,c))\in\Omega$ then $\psi(t) = \phi(t,c)$ for some $c\in\bar{\Omega}$
    \item The graph of $\phi$ as a function of $(t,c)$ is an $r$-dimensional manifold.
  \end{enumerate}
\end{definition}
\begin{definition}[Regular matrix pencil]
   \label{def:preliminaries-1-Regular-Matrix-Pencil}
  Let $\mat{A}$ and $\mat{B}$ be $N\times N$ matrices, then the \textnormal{matrix pencil} is defined as the matrix $\mat{A}+\lambda\mat{B}$ for some complex parameter $\lambda$. The matrix pencil $\mat{A}+\lambda\mat{B}$ is said to be \textnormal{regular} if its determinant, denoted $\det\left(\mat{A}+\lambda\mat{B}\right)$, is not identically zero as a function of $\lambda$.
\end{definition}
\begin{theorem}[DAE solvability and regularity of its pencil \cite{brenan96:_numer_solut_of_initial_value}]
  \label{def:preliminaries-1-Regular-means-solvable}
  The DAE system (\ref{eq:The-DAE}) is solvable if and only if the matrix pencil $\mat{G}+\lambda\mat{C}$ is regular.
\end{theorem}
\begin{theorem} [Weierstrass transform \cite{campbell80:_singul_system_differ_equat,kunkel06:_differ_algeb_equat_analy_and_numer_solut}]
  \label{them:Canonincal-forms}
  Suppose that matrices $\mat{A},\mat{B}\in\mathbb{R}^{q \times q}$ are real matrices and $\mat{A}+\lambda \mat{B}$ is a regular pencil, then there exist non-singular real matrices $\mat{P},\mat{Q}\in\mathbb{R}^{q\times q}$ such that
  \begin{equation}
    \label{eq:theorem-1-equation}
    \mat{P}\mat{A}\mat{Q} =
    \left[
      \begin{array}[c]{c c}
        \mat{I}_r & \mat{0}\\
        \mat{0} & \mat{N}
      \end{array}
    \right], \quad
    \mat{P}\mat{B}\mat{Q} =
    \left[
      \begin{array}[c]{c c}
        \mat{J} & \mat{0}\\
        \mat{0} & \mat{I}_s
      \end{array}
    \right],
  \end{equation}
  where $\mat{N}\in \mathbb{R}^{s\times s}$ is a nilpotent matrix with nilpotency index $k$, $\mat{J}\in\mathbb{R}^{r\times r}$ is a matrix in Jordan canonical form, with $r+s=q$. In case that $\mat{N}=\mat{0}$ then define $k=1$. In the special case that $\mat{A}$ is nonsingular, then take $\mat{P}\mat{A}\mat{Q}=\mat{I}_q$, $\mat{P}\mat{B}\mat{Q}=\mat{J}$ and define $k=0$. If $\mat{A}+\lambda \mat{B}$ is identically constant, then (\ref{eq:theorem-1-equation}) simplifies to $\mat{P}\mat{A}\mat{Q}=\mat{N}$ and $\mat{P}\mat{B}\mat{Q}=\mat{I}_q$. The matrix pair $\mat{P}$ and $\mat{Q}$ are referred to as the \textnormal{Weierstrass transform} of $\mat{A}$ and $\mat{B}$.
\end{theorem}
\begin{theorem}[Equivalence between differentiation and nilpotency index \cite{brenan96:_numer_solut_of_initial_value}]
  \label{them:diff-index-is-nilpotency-index}
  Suppose the DAE system in \eqref{eq:The-DAE} is solvable with differentiation index $\mu\geq 1$, and therefore there exist matrices $\mat{P}$ and $\mat{Q}$ satisfying \eqref{eq:theorem-1-equation}. Then the nilpotecy index of the resulting $\mat{N}$ matrix is equal to $\mu$.
\end{theorem}

\begin{lemma}[Order of the truncation \cite{gear71:_numer}]
  \label{lemma:Obreshkov-Order}
  Let $z(t)$ be a continuously differentiable function in $t$ and let $h=t_{n}-t_{n-1}$. Then,
  \begin{equation}
    \label{eq:Lemma-Obreshkov-error}
     \sum_{i=0}^m (-1)^i\alpha_{i,l,m} h^i\left. \frac{\mathrm{d}^iz(t)}{\mathrm{d}t^i}\right|_{t=t_n}- \sum_{i=0}^l \alpha_{i,m,l} h^i \left. \frac{\mathrm{d}^iz(t)}{\mathrm{d}t^i}\right|_{t=t_{n-1}} = \mathcal{O}\left(h^{l+m+1}\right)
  \end{equation}
\end{lemma}

\subsection{Overview of the Analysis}
The analysis presented in this section is developed along three main steps.
\begin{enumerate}
\item The first step will utilize a Weierstrass tranformation on the matrices $\mat{C}$ and $\mat{G}$ which decouples the DAE in (\ref{eq:The-DAE}) into two subsystems: the first takes the form of an ODE while the other subsystem is a purely algebraic subsystem. This step will result in decomposing the state $\mat{x}(t)$ into two components: the first is a solution to the ODE subsystem, which will be denoted $\mat{z}_{\mathrm{D}}(t)$, and the other is a solution to the algebraic subsystem which is denoted by $\mat{z}_{\mathrm{A}}(t)$.
\item In the second step, another Weierstrass tranformation is applied to discrete system matrices in (\ref{eq:The Augmented MNA for the Obreshov formula}) $\mattilde{C}$ and $\mattilde{G}$ decoupling it into two discrete systems and decomposing $\mathats{\xi}_n$ into two components, which will be denoted by $\mathats{\zeta}_{D,n}$ and  $\mathats{\zeta}_{A,n}$. It will be demonstrated that $\mathats{\zeta}_{D,n}$ and $\mathats{\zeta}_{A,n} $ carry the approximations to $\mat{z}_{\mathrm{D}}(t)$, and $\mat{z}_{\mathrm{A}}(t)$, and their derivatives, up to the $m^{\mathrm{th}}$ order derivatives (at $t=t_n$), to varying orders of the step size $h$.
  \item The third step will then link the entries of $\mathats{\zeta}_{D,n}$ and $\mathats{\zeta}_{A,n} $ to the entries in $\mathats{\xi}_n$ providing the path to describe the approximations generated by the Obreshkov method in $\mathats{\xi}_n$ in terms of the exact $\mat{x}(t_n)$ and the step size $h$.
\end{enumerate}

Moving forward, it will be assumed that the past time step is devoid of error, that is,
\begin{equation}
  \label{eq:past-time-step-is-accurate}
  h^i\mathat{x}_{n-1}^{(i)} = h^i\mat{x}^{(i)}\left(t_{n-1}\right),\quad i=0,1,\cdots,l
\end{equation}
so that the error derived error at $t=t_n$ reflects only the ``local'' approximation error committed in a single  time step. It will also be convenient to group all the exact derivatives scaled by powers of the step size $h$ in a single vector $\mats{\xi}\left(t\right)$, defined by
\begin{equation}
  \label{eq:Def-of-xi-of-t}
 \mats{\xi}(t)\coloneqq
                \left[
                \begin{array}[c]{c c c c}
                  \left(\mat{x}^{(0)}(t)\right)^{\top} & \left(h\mat{x}^{(1)}(t)\right)^{\top} &\cdots &\left(h^m\mat{x}^{(m)}(t)\right)^{\top}
                \end{array}
                \right]^{\top}  
\end{equation}

The previous steps are detailed in the following subsections.

\subsection{Step 1. The ODE and Algebraic Subsystems}
\label{sec:step1.-ode-algebraic}
Assuming the DAE (\ref{eq:The-DAE}) system is solvable, then it follows by Theorem \ref{them:Canonincal-forms} that are nonsingular matrices $\mat{P}$ and $\mat{Q}$ that can be used in a Weierstrass transform on $\mat{C}$ and $\mat{G}$. To this end, $\mat{Q}$ is used in following the change of variables $\mat{x}(t)\rightarrow\mat{z}(t)$ in (\ref{eq:The-DAE})
\begin{equation}
  \label{eq:change-ofvaiables-1}
  \mat{z}(t)\coloneqq \mat{Q}^{-1}\mat{x}(t),
\end{equation}
which, after pre-multiplying the DAE system (\ref{eq:The-DAE}) by $\mat{P}$ decouples it into (using (\ref{eq:theorem-1-equation}))
\begin{eqnarray}
  \label{eq:the decoupled MNA-diff-part}
  \frac{\textnormal{d}\mat{z}_D(t)}{\textnormal{d}t} &=& -\mat{J}\mat{z}_D(t)  + \mat{u}_{D}(t)\\
  \label{eq:the decoupled MNA-algeb-part}
  \mat{N}\frac{\textnormal{d}\mat{z}_A(t)}{\textnormal{d}t} &=& -\mat{z}_A(t)\phantom{\mat{J}} + \mat{u}_{A}(t)
\end{eqnarray}
where $\mat{z}_D(t),\mat{u}_D(t)\in\mathbb{R}^r$ and $\mat{z}_A(t),\mat{u}_A(t)\in\mathbb{R}^s$ that are obtained from 
\begin{equation}
  \label{eq:def-of-the-decoupled-system}
  \left[
    \begin{array}[c]{c c}
          \mat{u}_D(t)^{\top} & \mat{u}_A(t)^{\top}
    \end{array}
  \right]^{\top} = \mat{P}\mat{b}(t), \quad \quad
  \left[
    \begin{array}[c]{c c}
          \mat{z}_D(t)^{\top} & \mat{z}_A(t)^{\top}
    \end{array}
  \right]^{\top} =
  \mat{z}(t)
\end{equation}
It is obvious that (\ref{eq:the decoupled MNA-diff-part}) is a classical ODE whose solution is uniquely determined based on the initial condition at $t = 0$, and the stimulus $\mat{u}_D(t)$. However, the second part of (\ref{eq:the decoupled MNA-algeb-part}) is a purely algebraic part whose solution depends solely on the driving stimulus $\mat{b}(t)$ and its derivatives. This fact can be demonstrated when the solution to (\ref{eq:the decoupled MNA-algeb-part}) is written explicitly, as explained in \cite{brenan96:_numer_solut_of_initial_value}, in the following form
\begin{equation}
  \label{eq:exact-solution-of-the-algebraic}
  \mat{z}_A(t) = \sum_{i=0}^{k-1} (-1)^i\mat{N}^i \frac{\mathrm{d}^i}{\mathrm{d}t^i}\left(\mat{u}_A(t)\right) 
\end{equation}
indicating that $\mat{z}_A(t)$ for any time $t$ are determined solely from $\mat{b}(t)$ and its derivatives, at the same time $t$.

For convenience and later usage, let $\mat{Q}_D$ mark the fist $k$ columns of $\mat{Q}$ and $\mat{Q}_A$ its remaining ($s=N-k$) columns. Thus,
\begin{equation}
  \label{eq:z-of-t-as-function-of-z}
  \mat{x}(t) = \mat{Q}_D\mat{z}_D(t) + \mat{Q}_A\mat{z}_A(t)
\end{equation}

It would also be convenient for the purposes of the following analysis to stack the vectors $\mat{z}(t)$, $\mat{z}_D(t)$ and $\mat{z}_A(t)$ and their high-order derivatives (scaled by powers of $h$) in single vectors $\mats{\zeta}(t)$, $\mats{\zeta}_D(t)$, $\mats{\zeta}_A(t)$ defined, respectively, by
\begin{eqnarray}
  \label{eq:Def-of-xi-and-zeta}
  \mats{\zeta}(t)&\coloneqq&
                \left[
                \begin{array}[c]{c c c c}
                  \left(\mat{z}^{(0)}(t)\right)^{\top} & \left(h\mat{z}^{(1)}(t)\right)^{\top} &\cdots &\left(h^m\mat{z}^{(m)}(t)\right)^{\top}
                \end{array}
                \right]^{\top}\\
  \label{eq:Def-of-zeta-D}
  \mats{\zeta}_D(t)&\coloneqq&
                \left[
                \begin{array}[c]{c c c c}
                  \left(\mat{z}_D^{(0)}(t)\right)^{\top} & \left(h\mat{z}_D^{(1)}(t)\right)^{\top} &\cdots &\left(h^m\mat{z}_D^{(m)}(t)\right)^{\top}
                \end{array}
                                                                                                             \right]^{\top}\\
  \label{eq:Def-of-zeta-A}
               \mats{\zeta}_A(t)&\coloneqq&
                \left[
                \begin{array}[c]{c c c c}
                  \left(\mat{z}_A^{(0)}(t)\right)^{\top} & \left(h\mat{z}_A^{(1)}(t)\right)^{\top} &\cdots &\left(h^m\mat{z}_A^{(m)}(t)\right)^{\top}
                \end{array}
                \right]^{\top}
\end{eqnarray}
Using the Kronecker (tensor) operator $\otimes$, it is obvious from \eqref{eq:Def-of-xi-of-t} and the above that
\begin{equation}
  \label{eq:xi-as-function-of-zeta}
  \mats{\xi}(t) = \left(\mat{I}_{m+1}\otimes \mat{Q}_D\right)\mats{\zeta}_D(t) + \left(\mat{I}_{m+1}\otimes \mat{Q}_A\right)\mats{\zeta}_A(t)
\end{equation}
\subsection{Step 2. Applying the Weierstrass transform on the Discrete System}
\label{sec:step-2.-applying}
In this step, the matrices $\mattilde{P},\mattilde{Q}\in\mathbb{R}^{(m+1)N\times (m+1)N}$ defined by
\begin{equation}
  \label{eq:def-of-P-tildeand Q-tilde}
  \mattilde{P} \coloneqq
  \left[
    \begin{array}[c]{c c c c c }
      \mat{P} & \mat{0} & \cdots  &        & \mat{0}\\
      \mat{0} & \mat{P} & \mat{0} & \cdots & \mat{0}\\
      \vdots  & \vdots  & \ddots  &        & \mat{\vdots}\\
      \mat{0} & \mat{0} &  \cdots & \mat{P}& \mat{0}\\
      \mat{0} & \mat{0} &         & \cdots & \mat{Q}^{-1}
    \end{array}
  \right],\quad
  \mattilde{Q} \coloneqq
    \left[
      \begin{array}[c]{c c c c c }
        \mat{Q} & \mat{0} & \cdots  &        & \mat{0}\\
        \mat{0} & \mat{Q} & \mat{0} & \cdots & \mat{0}\\
        \vdots  & \vdots  & \ddots  &        & \mat{\vdots}\\
        \mat{0} & \mat{0} &  \cdots & \mat{Q}& \mat{0}\\
        \mat{0} & \mat{0} &         & \cdots & \mat{Q}
      \end{array}
    \right]
  \end{equation}
  are used to Weierstrass transform the matrices $\mattilde{C}$ and $\mattilde{G}$. This is carried out through performing the change of variables $\mathats{\xi}_n\rightarrow\mathats{\zeta}_n$ given by
  \begin{eqnarray}
    \label{eq:augmented-change-of-vars}
    \mathats{\zeta}_{n} &\coloneqq& \mattilde{Q}^{-1}\mathats{\xi}_{n}\nonumber\\
                     &=& \left[
                         \begin{array}[c ]{c c c c}
                           \left(\mat{Q}^{-1} \mathat{x}_n^{(0)}\right) ^{\top}&
                           \left(h\mat{Q}^{-1} \mathat{x}_n^{(1)}\right) ^{\top}&
                                                                                                                            \cdots &
                            \left(h^m\mat{Q}^{-1} \mathat{x}_n^{(m)}\right) ^{\top}
                         \end{array}\right]^{\top}
  \end{eqnarray}
  Define vectors $\mathat{z}_n^{(i)}\in\mathbb{R}^N$, $\mathat{z}_{n-1}^{(i)}\in\mathbb{R}^N$
  \begin{eqnarray}
    \label{eq:def-of-z-hat-n}
    \mathat{z}_n^{(i)}& \coloneqq &\mat{Q}^{-1} \mathat{x}_n^{(i)},\quad i=0,\cdots,m\\
    \label{eq:def-of-z-hat-n-1}
    \mathat{z}_{n-1}^{(i)}& \coloneqq &\mat{Q}^{-1} \mathat{x}_{n-1}^{(i)},\quad i=0,\cdots,l
  \end{eqnarray}
  and partition them into vectors of sizes $r$ and $s$,
  \begin{equation}
    \label{eq:partioniong-of-z-hat}
    \mathat{z}_n^{(i)} = \left[{\mathat{z}_{D,n}^{(i)}}^{\top} \quad {\mathat{z}_{A,n}^{(i)}}^{\top}\right]^{\top},\quad \mathat{z}_{n-1}^{(i)} = \left[{\mathat{z}_{D,n-1}^{(i)}}^{\top} \quad {\mathat{z}_{A,n-1}^{(i)}}^{\top}\right]^{\top}
  \end{equation}
  
  Next the system (\ref{eq:The Augmented MNA for the Obreshov formula}), is pre-multiplied by the matrix $\mattilde{P}$, which along with the change of variables in (\ref{eq:augmented-change-of-vars}) decouples it into the following two systems
\begin{eqnarray}
  \label{eq:decoupled-discretized-differential-part-diff}
  \mat{K}\mathats{\zeta}_{D,n} &=&  \mat{e}_{D,n}\\
  \label{eq:decoupled-discretized-differential-part-algeb}
  \mat{M}\mathats{\zeta}_{A,n} &=&  \mat{e}_{A,n},
\end{eqnarray}
where the vectors $\mathats{\zeta}_{D,n} \in\mathbb{R}^{(m+1)r}$ and $\mathats{\zeta}_{A,n} \in\mathbb{R}^{(m+1)s}$ group the vectors $\mathat{z}_{D,n}^{(i)}$ and $\mathat{z}_{A,n}^{(i)}$, respectively, i.e.,
\begin{eqnarray}
  \label{eq:Augmented-transformed}
  \mathats{\zeta}_{D,n} &=&  {\left[
    \begin{array}[c]{c c c c}
      {\mathat{z}^{(0)}_{D,n}}^{\top} & h{\mathat{z}^{(1)}_{D,n}}^{\top}& \cdots & h^m{\mathat{z}^{(m)}_{D,n}}^{\top}
    \end{array}
                                                                                   \right]}^{\top}\\
  \label{eq:Augmented-transformed-Alg}
  \mathats{\zeta}_{A,n} &=&  {\left[
      \begin{array}[c]{c c c c}
      {\mathat{z}^{(0)}_{A,n}}^{\top} & h{\mathat{z}^{(1)}_{A,n}}^{\top}& \cdots & h^m{\mathat{z}^{(m)}_{A,n}}^{\top}
    \end{array}
                                                                                   \right]}^{\top}
\end{eqnarray}

Moreover the matrices $\mat{K}\in\mathbb{R}^{(m+1)r\times (m+1)r} $ and $\mat{M}\in\mathbb{R}^{(m+1)s\times (m+1)s} $, and the vectors $ \mat{e}_{D,n} \in\mathbb{R}^{(m+1)r}, \mat{e}_{A,n} \in\mathbb{R}^{(m+1)s}$ are, respectively, given by
{\small{
\begin{eqnarray}
  \label{eq:def of K }
  \mat{K} &=&
  \left[
    \begin{array}[c]{c c c c c }
      \mat{J}     & \frac{1}{h}\mat{I}_r     & \cdots  &        & \mat{0}\\
      \mat{0}     & \mat{J}     & \frac{1}{h}\mat{I}_r & \cdots & \mat{0}\\
      \vdots      & \vdots      & \ddots  &        & \mat{0}\\
      \mat{0}     & \mat{0}     & \cdots  & \mat{J}& \frac{1}{h}\mat{I}_r\\
      \alpha_{0,l,m}\mat{I}_r & -\alpha_{1,l,m}\mat{I}_r &         & \cdots & (-1)^{m}\alpha_{m,l,m}\mat{I}_r
    \end{array}
  \right]\\
  \label{eq:def of M}
   \mat{M} &=&
  \left[
    \begin{array}[c]{c c c c c }
      \mat{I}_s     & \frac{1}{h}\mat{N}     & \cdots  &        & \mat{0}\\
      \mat{0}     & \mat{I}_s     & \frac{1}{h}\mat{N} & \cdots & \mat{0}\\
      \vdots      & \vdots      & \ddots  &        & \mat{0}\\
      
      \mat{0}     & \mat{0}     & \cdots  & \mat{I}_s& \frac{1}{h}\mat{N}\\
      \alpha_{0,l,m}\mat{I}_s & -\alpha_{1,l,m}\mat{I}_s &         & \cdots & (-1)^{m}\alpha_{m,l,m}\mat{I}_s
    \end{array}
  \right]\\
  \label{eq:def-of-eD}\\
  \mat{e}_{D,n} &=&
   \left[
    \begin{array}[c]{c c c c }
      \left(\mat{u}_D^{(0)}(t_{n})\right)^{\top}&
                                                  \cdots&
                                                          \left(h^{m-1}\mat{u}_D^{(m-1)}(t_{n}) \right)^{\top}&
                                                                                                                \left(\sum_{i=0}^{l}\alpha_{i,m,l}h^{i}\mathat{z}_{D,n-1}^{(i)}\right)^{\top}
    \end{array}
                                                                                                                \right]^{\top
                                                                                                                },\nonumber\\
  \label{eq:def-of-eA}\\
  \mat{e}_{A,n} &=&
   \left[
    \begin{array}[c]{c c c c }
      \left(\mat{u}_A^{(0)}(t_{n})\right)^{\top}&
                                                  \cdots&
                                                          \left(h^{m-1}\mat{u}_A^{(m-1)}(t_{n}) \right)^{\top}&

                                                                                                                \left(\sum_{i=0}^{l}\alpha_{i,m,l}h^{i}\mathat{z}_{A,n-1}^{(i)}\right)^{\top}
    \end{array}
                                                                                                                \right]^{\top}\nonumber
\end{eqnarray}
}} 
In (\ref{eq:def-of-eD}) and (\ref{eq:def-of-eA}), the vectors $\mathat{z}_{D,n-1}^{(i)}$ $\mathat{z}_{A,n-1}^{(i)}$ are, $r$ and $s$ partitions, respectively, of the vector $\mathat{z}_{n-1}^{(i)}$, which is given by $\mathat{z}_{n-1}^{(i)} = \mat{Q}^{-1}\mathat{x}_{n-1}^{(i)} $, $i=0,\cdots,l$.
Similarly to \eqref{eq:xi-as-function-of-zeta}, $\mathats{\xi}_n $ can be expressed as,
\begin{equation}
  \label{eq: Def-of-xi-hat-and-zeta-hat}
  \mathats{\xi}_n = \left(\mat{I}_{m+1}\otimes \mat{Q}_D\right)\mathats{\zeta}_{D,n} + \left(\mat{I}_{m+1}\otimes \mat{Q}_A\right)\mathats{\zeta}_{A,n}
\end{equation}

\subsection{Step 3. Derivation of the Obreshkov Order of Approximation}
\label{sec:step-3.-derivation}
The task of characterizing the order of the Obreshkov method can be accomplished if the difference between the -exact- $\mats{\xi}\left(t_{n}\right)$ and the -Obreshkov-approximated- $\mathats{\xi}_n$ is derived showing its relation to the step size $h$. Using \eqref{eq:xi-as-function-of-zeta} and \eqref{eq: Def-of-xi-hat-and-zeta-hat}, this difference is given by
\begin{equation}
  \label{eq:difference-between-xi-and-xi-hat}
  \mathats{\xi}_n  - \mats{\xi}\left(t_n\right)  = \left(\mat{I}_{m+1}\otimes \mat{Q}_D\right)\left(\mathats{\zeta}_{D,n} - \mats{\zeta}_D\left(t_n\right) \right)+
   \left(\mat{I}_{m+1}\otimes \mat{Q}_A\right)\left(\mathats{\zeta}_{A,n} - \mats{\zeta}_A\left(t_n\right) \right)
\end{equation}
indicating that the approximation error ($\mathats{\xi}_n -\mats{\xi}\left(t_{n}\right)$) of the Obreshkov method results from two different components. The first component is the differential component which arises from $\mathats{\zeta}_{D,n} - \mats{\zeta}_D\left(t_n\right)$ while the other component is the one resulting from the algebraic component $\mathats{\zeta}_{A,n} - \mats{\zeta}_A\left(t_n\right)$. Characterizing each of those components is considered separately in the following subsections.

\subsubsection{The Algebraic Error Component}
\label{sec:algebr-comp-error}
Using \eqref{eq:the decoupled MNA-algeb-part}, multiplying by $h^i$ and differentiating both sides $i$ times with respect to $t$ yields
\begin{equation}
  \label{eq:proof-theorem-1}
  \frac{1}{h}\mat{N}\frac{\mathrm{d}^{i+1}}{\textnormal{d}t^{i+1}}\left(h^{i+1}\mat{z}_A(t) \right)= -\frac{\textnormal{d}^{i}}{\textnormal{d}t^{i}}\left(h^i\mat{z}_A(t) \right)+ \frac{\textnormal{d}^{i}}{\textnormal{d}t^{i}}\left(h^i\mat{u}_A(t)\right), \quad i = 0,\cdots, m-1
\end{equation}
Also, from Lemma \ref{lemma:Obreshkov-Order}, we have
\begin{equation}
  \label{eq:Obreshkov-order-applied}
   \sum_{i=0}^m (-1)^i\alpha_{i,l,m} h^i\left. \frac{\mathrm{d}^i\mat{z}_A(t)}{\mathrm{d}t^i}\right|_{t=t_n}=\sum_{i=0}^l \alpha_{i,m,l} h^i \left. \frac{\mathrm{d}^i\mat{z}_A(t)}{\mathrm{d}t^i}\right|_{t=t_{n-1}} +  \matcal{O}_s\left(h^{l+m+1}\right)
 \end{equation}

 The $m$ systems in \eqref{eq:proof-theorem-1} (at $t=t_n$) along with the system in \eqref{eq:Obreshkov-order-applied} can be put in a matrix form
 \begin{equation}
   \label{eq:System-of-zeta_n}
   \mat{M}\mats{\zeta}_A\left(t_n\right) = \mat{\Psi}
 \end{equation}
 where $\mat{M}$ is the matrix defined by \eqref{eq:def of M}, $\mats{\zeta}_A\left(t_n\right)$ is defined in \eqref{eq:Def-of-zeta-A} and the vector $\mats{\Psi}$ is given by

 \begin{equation}
   \label{eq:Def-of-Psi}
   \mat{\Psi}\coloneqq\left[ \begin{array}[c]{c  c c c}
      {\mat{u}_A^{(0)}(t_{n})}^{\top} &
                               \cdots&
                               {h^{m-1}\mat{u}_A^{(m-1)}(t_{n}) }^{\top} &
                                                                           \left(
                                                                           \begin{array}[c]{c}
                                                                             \sum_{i=0}^l\alpha_{i,m,l} h^i \mat{z}_A^{(i)}\left(t_{n-1}\right)\\
                                                                             +\matcal{O}_s\left(h^{l+m+1}\right)
                                                                           \end{array}
                                                                           \right)^{\top}
                             \end{array}
                           \right]^{\top}
                         \end{equation}

 Subtraction of \eqref{eq:System-of-zeta_n} from \eqref{eq:decoupled-discretized-differential-part-algeb}, and noting from \eqref{eq:past-time-step-is-accurate} and \eqref{eq:def-of-z-hat-n-1} that $\mathat{z}_{A,n-1}^{(i)}=\mat{z}_A^{(i)}\left(t_{n-1}\right)$, yields
 \begin{equation}
   \label{eq:2Algebraic-error}
   \mathats{\zeta}_{A,n}-\mats{\zeta}_A\left(t_n\right) = \mat{M}^{-1}\mats{\Delta}
 \end{equation}
 where
 \begin{equation}
   \label{eq:Def-delta}
   \mats{\Delta} = \left[\mat{0}, \,\cdots \,\mat{0},\, \matcal{O}_s\left(h^{l+m+1}\right)^{\top}\right]^{\top}
 \end{equation}
 The inverse of the matrix $\mat{M}$ can be expressed in block-structured format by partitioning it into four blocks,
 \begin{equation}
  \label{eq:block-partitionof-Mtilde}
  \mat{M} =
  \left[
    \begin{array}[c ]{c c}
      \mat{M}_{11} & \mat{M}_{12}\\
      \mat{M}_{21} & \mat{M}_{22}
    \end{array}
  \right]
\end{equation}
 where,
\begin{eqnarray*}
  \label{eq:M_11}
  \mat{M}_{11} &=&
  \left[
    \begin{array}[c]{c c c c c}
      \mat{I}_s & \frac{1}{h}\mat{N} & \mat{0} & \cdots & \mat{0}\\
      \mat{0} & \mat{I}_s & \frac{1}{h}\mat{N}& & \vdots\\
      \mat{0} &         & \ddots & &       \\
      \vdots  &         & \ddots & & \frac{1}{h}\mat{N}\\
      \mat{0} & \mat{0} & \cdots & & \mat{I}_s\\
    \end{array}
  \right],\quad
  \mat{M}_{12} =
  \left[
    \begin{array}[c]{c}
      \mat{0}\\
      \vdots \\
      \\
      \mat{0}\\
      \frac{1}{h}\mat{N}
    \end{array}
  \right]\\
   \mat{M}_{21} &=&
  \left[
    \begin{array}[c]{c c c }
      \alpha_{0,l,m}\mat{I}_s & \cdots  & (-1)^{m-1}\alpha_{m-1,l,m}\mat{I}_s \\
    \end{array}
  \right],\\
  \mat{M}_{22} &=& (-1)^m\alpha_{m,l,m}\mat{I}_s
\end{eqnarray*}
The inverse of the block-partitioned matrix (\ref{eq:block-partitionof-Mtilde}) is given by the block-partitioned matrix
\begin{equation}
     \label{eq:Inverse-block-partitioned}
     \mat{M}^{-1} =
     \left[
       \begin{array}[c]{c c}
         \mat{M}_{11}^{-1} + \mat{M}_{11}^{-1}\mat{M}_{12}\mat{S}_A^{-1}\mat{M}_{21}\mat{M}_{11}^{-1} & -\mat{M}_{11}^{-1}\mat{M}_{12}\mat{S}_A^{-1}\\
         \mat{S}_A^{-1}\mat{M}_{21}\mat{M}_{11}^{-1}                                                  &\mat{S}_A^{-1}
       \end{array}
     \right]
   \end{equation}
   with
   \begin{equation}
     \label{eq:S_A-appendix}
     \mat{S}_A = \mat{M}_{22} - \mat{M}_{21}\mat{M}_{11}^{-1}\mat{M}_{12}
   \end{equation}

   The error in algebraic component is therefore given by,
   \begin{equation}
     \label{eq:error-algebraic-comp}
     \mathats{\zeta}_{A,n}-\mats{\zeta}_A\left(t_n\right) =
     \left[
       \begin{array}[c]{c }
         -\mat{M}_{11}^{-1}\mat{M}_{12}\mat{S}_A^{-1}\\
         \mat{S}_A^{-1}
       \end{array}
     \right]\matcal{O}_s\left(h^{l+m+1}\right)
   \end{equation}

   Given that $\mat{M}_{11}$ is a block-upper diagonal matrix, with identity matrices on the diagonal blocks, its inverse is trivial and is given by
\begin{equation}
  {\small{
  \label{eq:inverse-of-M-11}
  \mat{M}_{11}^{-1} =
  \left[
     \begin{array}[c]{c c c c c c }
      \mat{I}_s & \frac{-1}{h}\mat{N}&  \frac{1}{h^2}\mat{N}^2&  \frac{-1}{h^3}\mat{N}^3  & \cdots &   \frac{(-1)^{m-1}}{h^{m-1}}\mat{N}^{m-1}\\
      \mat{0} & \mat{I}_s           &  \frac{-1}{h}\mat{N}    &  \frac{1}{h^2}\mat{N}^2  & \cdots &  \frac{(-1)^{m-2}}{h^{m-2}}\mat{N}^{m-2}\\
      \mat{0} &                    & \ddots                 &  \ddots                &        &                                \\
      \vdots  &                    &                        &       \ddots           &        &              \frac{-1}{h}\mat{N}\\
      \mat{0} & \mat{0}            & \cdots                 &                        &        & \mat{I}_s\\
    \end{array}
  \right]}}
\end{equation}
$\mat{S}_A$ is expanded as  
\begin{equation}
  \label{eq:S-A}
  \mat{S}_A = \sum_{i=0}^{m} \alpha_{i,l,m} (-1)^i  \left(\frac{\mat{N}}{h}\right)^{m-i}
\end{equation}
and its inverse $\mat{S}_A^{-1}$ can be expanded in a Taylor series format
\begin{equation}
  \label{eq:S-A-inverse}
  \mat{S}_A^{-1} = \sum_{p=0}^{\infty}\gamma_p   \left(\frac{\mat{N}}{h}\right)^p
\end{equation}
where $\gamma_p$ denote the sum of all the coefficients that appear in front of the terms with  $\left(\frac{\mat{N}}{h}\right)^p$.
Substituting \eqref{eq:S-A-inverse} and \eqref{eq:inverse-of-M-11} into \eqref{eq:error-algebraic-comp} and using the definition of $\mat{M}_{12}$ yields
  \begin{equation}
     \label{eq:error-algebraic-comp-expanded}
     \mathats{\zeta}_{A,n}-\mats{\zeta}_A\left(t_n\right) =
     \left[
       \begin{array}[c]{c }
         (-1)^{m}\sum_{p=0}^{\infty}\gamma_p   \left(\frac{\mat{N}}{h}\right)^{p+m}\\
         \vdots\\
         (-1)^{m-i}\sum_{p=0}^{\infty}\gamma_p   \left(\frac{\mat{N}}{h}\right)^{p+m-i}\\
         \vdots\\
         -1\sum_{p=0}^{\infty}\gamma_p   \left(\frac{\mat{N}}{h}\right)^{p+1}\\
       \sum_{p=0}^{\infty}\gamma_p   \left(\frac{\mat{N}}{h}\right)^p
       \end{array}
     \right]\matcal{O}_s\left(h^{l+m+1}\right)
   \end{equation}

   Let us now assume that the differentiation index of the DAE system is $k$. Thus from Theorem \ref{them:diff-index-is-nilpotency-index} the matrix $\mat{N}$ is nilpotent with nilpotency index $k$. Therefore, $\mat{N}^{q}=\mat{0}$ for $q \geq k$. This fact makes the algebraic component of the error in the $i$-th order derivative, (that is $\mathat{z}_{A,n}^{(i)}-\mat{z}_A^{(i)}(t_n), i=0,\cdots,m$), vanish completely if $m-i\geq k$. On the other hand, the case  $m-i < k$  entails truncating the infinite series in \eqref{eq:error-algebraic-comp-expanded} (in accordance with $\mat{N}^{q}=\mat{0}$ for $q\geq k$) at $p=k-m+i-1$ without loss of accuracy. Thus we have,
  \begin{equation}
     \label{eq:error-algebraic-comp-expanded-2}
     \mathats{\zeta}_{A,n}-\mats{\zeta}_A\left(t_n\right) =
     \left[
       \begin{array}[c]{c }
         (-1)^{m}\sum_{p=0}^{k-m-1}\gamma_p   \left(\frac{\mat{N}}{h}\right)^{p+m}\\
         \vdots\\
         (-1)^{m-i}\sum_{p=0}^{k-m+i-1}\gamma_p   \left(\frac{\mat{N}}{h}\right)^{p+m-i}\\
         \vdots\\
         -1\sum_{p=0}^{k-2}\gamma_p   \left(\frac{\mat{N}}{h}\right)^{p+1}\\
         \sum_{p=0}^{k-1
         }\gamma_p   \left(\frac{\mat{N}}{h}\right)^p
       \end{array}
     \right]\matcal{O}_s\left(h^{l+m+1}\right)
   \end{equation}
   The above expression enable deducing the order of the error through finding the \textit{smallest positive} integer that appears on the power of $h$ in each component. For example, in the first component the smallest positive power of $h$ appears at $p=k-m-1$, giving rise to an order of $l+m+2-k$ Using analogous reasoning leads to the following result
   \begin{equation}
     \label{eq:error-algebraic-comp-expanded-4}
     \mathat{z}_{A,n}^{(i)}-\mat{z}_A^{(i)}(t_n) =
     \left\lbrace
       \begin{array}[c]{c l }
         \matcal{O}_s\left(h^{l+m+2 \,- k}\right) &\mathrm{if}\, m-i <k\\
         \mat{0}&   \mathrm{if}\, m-i\geq k
       \end{array}
     \right.
   \end{equation}
   
   \subsubsection{The Differential Error Component}
   \label{sec:diff-error-comp}
   This part of the error analysis also lays out the proof of Lemma \ref{lemma:order-obreshk-form--ODE} since the system of \eqref{eq:the decoupled MNA-diff-part} is in the ODE form, and its error resembles the error of a general ODE addressed in that lemma.

   The error in the differential component can be derived in a similar manner to the algebraic component. Starting with \eqref{eq:proof-theorem-1} and \eqref{eq:Obreshkov-order-applied}, replacing $\mat{z}_A(t)$ for $\mat{z}_D(t)$, yields the following
   \begin{equation}
     \label{eq:Differential-error-component}
     \mathats{\zeta}_{D,n}-\mats{\zeta}_D\left(t_n\right) = \mat{K}^{-1} \mats{\Delta}
   \end{equation}
   where $\mat{K}$ is given by \eqref{eq:def of K } and $\mathats{\zeta}_{D,n}$ and $\mats{\zeta}_{D}\left(t_n\right)$ are, respectively, defined in~\eqref{eq:Augmented-transformed} and \eqref{eq:Def-of-zeta-D}. A process similar to the process of deriving the algebraic error can be followed if, in the matrix $\mat{M}$,  $\mat{N}$ is replaced with $\mat{I}_r$ above the diagonal, $\mat{I}_s$ is replaced with $\mat{J}$ on the main diagonal and $\mat{I}_s$ is replaced by $\mat{I}_r$ on the last block of rows. Using the formula derived above for the inverse of a $2 \times 2$ block-partitioned matrix, the error in the diferential component results in
   \begin{equation}
     \label{eq:Differential-error-component-1}
     \mathats{\zeta}_{D,n}-\mats{\zeta}_D\left(t_n\right) =
     \left[
       \begin{array}[c]{c}
         \frac{h^m}{h^{m\phantom{-1}}}\left\lgroup\sum_{i=0}^{m}(-1)^i\alpha_{i,l,m}h^i\left(\mat{J}^{-1}\right)^{m-i}\right\rgroup^{-1}\mat{J}\\
         \frac{h^m}{h^{m-1}}\left\lgroup\sum_{i=0}^{m}(-1)^i\alpha_{i,l,m}h^i\left(\mat{J}^{-1}\right)^{m-i}\right\rgroup^{-1}\mat{J}\\
         \vdots\\
          \frac{h^m}{h^{0\phantom{-1}}}\left\lgroup\sum_{i=0}^{m}(-1)^i\alpha_{i,l,m}h^i\left(\mat{J}^{-1}\right)^{m-i}\right\rgroup^{-1}\mat{J}
       \end{array}
     \right]
     \matcal{O}_r\left(h^{l+m+1}\right)
   \end{equation}
   The above expressions can be used to deduce that the order of convergence in the $i^{\mathrm{th}}$ order derivative will be given by
   \begin{equation}
     \label{eq:Oder-in-in-differential}
     \mathat{z}_{D,n}^{(i)} - \mat{z}_D^{(i)}(t_n) =\matcal{O}_r\left(h^{l+m+1+i}\right) 
   \end{equation}
   \subsection{Main Result}
   \label{sec:final-theorem}
   The preceding analysis enables establishing the order of the Obreshkov method in a general DAE using the following theorem.
   \begin{theorem}
     \label{them:DAE-theorem-1}
     Let the DAE system in \eqref{eq:The-DAE} be given with a differentiation index $k >0$ and assume that the Obreshkov method with parameters $l,m$ is used to approximate $\mat{x}(t)$. Furthermore assume that $\mat{x}(t)$ and $\frac{\mathrm{d}^i\mat{x}(t)}{\mathrm{d}t^i}$ $i=0,1,\cdots,l$ are readily available at $t=t_{n-1}$ and have been assigned to $\mathat{x}^{(i)}_{n-1},i=0,\cdots,l$. It then follows that the approximations $\mathat{x}_n^{(i)}$ converge asymptotically to $\frac{\mathrm{d}^i\mat{x}(t)}{\mathrm{d}t^i}$ at $t=t_n$ with the following order
     \begin{equation}
       \label{eq:Oder-theorem-DAE}
      h^i \mathat{x}_n^{(i) } - h^i\mat{x}^{(i)}(t_n) = \left\lbrace
         \begin{array}[c]{l c}
           \matcal{O}_N\left(h^{l+m+2-k}\right) & \mathrm{if}\, m-i <k \\
            \matcal{O}_N\left(h^{l+m+1+i}\right) & \mathrm{if}\, m-i \geq k
         \end{array}
         \right.
       \end{equation}
   \end{theorem}
   \begin{proof}
     Substituting \eqref{eq:error-algebraic-comp-expanded-2}-\eqref{eq:error-algebraic-comp-expanded-4} and \eqref{eq:Differential-error-component}-\eqref{eq:Differential-error-component-1} into \eqref{eq:difference-between-xi-and-xi-hat}, letting $h\rightarrow 0$ and observing that the dominant component of the error $\mathats{\xi}_n-\mats{\xi}(t_n)$ is associated with the smallest positive power of $h$ proves the above theorem.
   \end{proof}
   \begin{remark}
     \label{rem:main-result}
     The local order of convergence obtained from the Obreshkov method with parameters $l,m$ when applied in DAE with index $k$ matches the same order convergence if the method is applied in ODE, if and only if $m\geq k$. Equivalently put, the Obreshkov method suffers order reduction if it is used in DAE with differentiation index $k > m$. 
   \end{remark}
   \section{Experimental Validation}
   \label{sec:exper-valid}

   Numerical validation of the theoretical results presented above requires problems where the exact solution (solution free from truncation error) of the DAE is obtainable. This is generally difficult since problems modelled by systems of DAE do not have their analytical solutions readily available. Fortunately, this problem can be handled in the domain of circuit simulation using the following steps.
   \begin{itemize}
   \item The circuit is excited by sinusoidal sources. In this case the source vector is given in the form $\mat{b}(t)=\mat{b}_c\cos\left(\omega t\right) + \mat{b}_s\sin\left(\omega t\right) $, with $\omega$ denoting the radial frequency in rad/sec and $\mat{b}_c$ and $\mat{b}_s$ are constant vectors.
   \item With sinusoidal stimulus at the input, the circuit response at steady-state ($t\rightarrow \infty$) settles down to a sinusoidal waveform represented by $\mat{x}_{\mathrm{ss}}(t)\coloneqq \mat{X}_c\cos\left(\omega t\right) + \mat{X}_s\sin\left(\omega t\right) $, where $\mat{X}_c$ and $\mat{X}_S$ are constant vectors that can be computed using the AC analysis method \cite{vlach-chapter-10}. This approach for computing the response of the circuit, or the solution of the DAE that model circuit formulation, in steady-state is indeed free from the truncation error that is associated with the methods that solve the DAE as an IVP. Therefore, it can be used as the accurate reference against which results from any such IVP methods can be compared. 
   \item Next, the initial value, $\mathat{x}_0$, used in starting the Obreshkov method is taken from an arbitrary point in the periodical trajectory traced by $\mat{x}_{\mathrm{ss}}(t)$. For example, the point $t=0$ is possible choice. Hence, $\mathat{x}_0$ can be assigned the value of $\mat{x}(0)=\mat{X}_C$. In a similar manner the initial values of the derivatives can also be assigned from the derivatives of the steady-state response $\mat{x}_{\mathrm{ss}}(t)$ at $t=0$. Thus, $\mathat{x}_0^{(1)} = \omega\mat{X}_S$, $\mathat{x}_0^{(2)} = -\omega^2\mat{X}_C$, and so forth.
   \item Running the Obreshkov method using specified values for parameters $l,m$, and starting with the point computed in the previous step should generate a sequence of points $\mathat{x}_n$ that approximate $\mat{x}_{\mathrm{ss}}(t)$. In fact, $\mathat{x}_1-\mat{x}_{\mathrm{ss}}(h)$, for sufficiently small values of $h$, should asymptotically approach $\matcal{O}(h^q)$ where $q$ is the power described by Theorem \ref{them:DAE-theorem-1} in \eqref{eq:Oder-theorem-DAE}. Thus, the validation of Theorem \ref{them:DAE-theorem-1} can be carried out by examining the behaviour of the error $||\mathat{x}_1-\mat{x}_{\mathrm{ss}}(h)||$ versus $h$.
     \item In order to clearly display the results, the error $||\mathat{x}_1-\mat{x}_{\mathrm{ss}}(h)||$ will be plotted on log-scaled graph versus the values of $h$. Naturally, if the error asymptotically approaches $h^q$ (as it should) then the log-plot will demonstrate a linear behaviour whose slope\footnote{Slope on a log scale graph is defined as the number of decades of increase/decrease in error $\mathat{x}_1-\mat{x}_{\mathrm{ss}}(h)$ per one decade change in $h$} should match the value described by \eqref{eq:Oder-theorem-DAE} in accordance with the values of $l,m$ and the index of the DAE $k$.
     
     \end{itemize}

     The above steps were executed on the three circuits shown in Figure \ref{fig:circuits}. The differentiation index of those circuits are, from left to right, $k=1,2,3$, respectively. The determination of the indices of these circuits was done using the procedure described by \cite{Gerdin:04b} to compute the matrices $\mat{P}$  and $\mat{Q}$ of the Weierstrass transformation. The $\mat{N}$ matrix resulting from the transformation was used to determine the differentiation index $k$ (using $\mat{N}^k=\mat{0}$) of the DAE modelling each circuit.
   \begin{figure}[htbp!]
     \centering
     \includegraphics[width=0.99\textwidth]{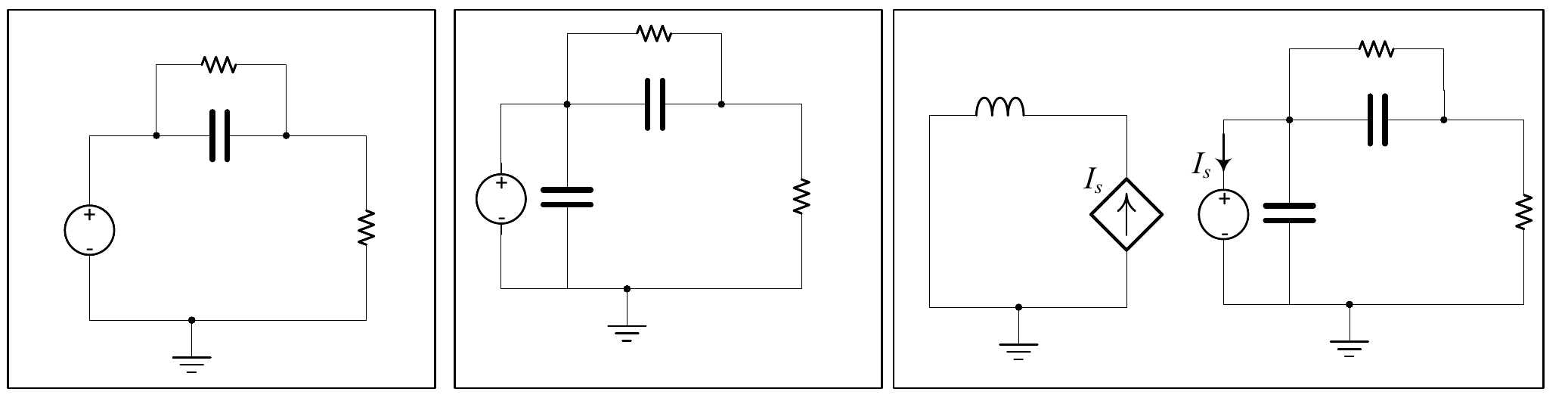}
     \caption{Circuits used in the numerical validation of the theoretical results. All resistors are equal $1\Omega$, capacitors  1F and inductors are 1H. Independent voltage source is $\cos(2\pi t) +\sin(2 \pi t)$. }
     \label{fig:circuits}
   \end{figure}

   The plots in Figure \ref{fig:affinity} display the error $||h^i\mathat{x}_1^{(i)}-h^i\mat{x}_{\mathrm{ss}}^{(i)}(h)||$ for $i=0, 1$ versus values of $h$ logarithmically distributed within one decade. Each plot indicates the values of $l,m$ used with the Obreshkov method and the differentiation index $k$ of the DAE system. The plots also highlight the slopes observed in each line. As shown, the slope in each case matches the order predicted by Theorem \ref{them:DAE-theorem-1} given the values of $l,m$ and $k$.
   \begin{figure}[t]
\setlength{\abovecaptionskip}{-5pt}
\setlength{\belowcaptionskip}{-5pt}
\addtolength{\subfigcapskip}{3pt}

\centering
\begin{tabular}{cc}
\subfigure[$l=0,m=2,k=2$]{\includegraphics[scale=0.36]{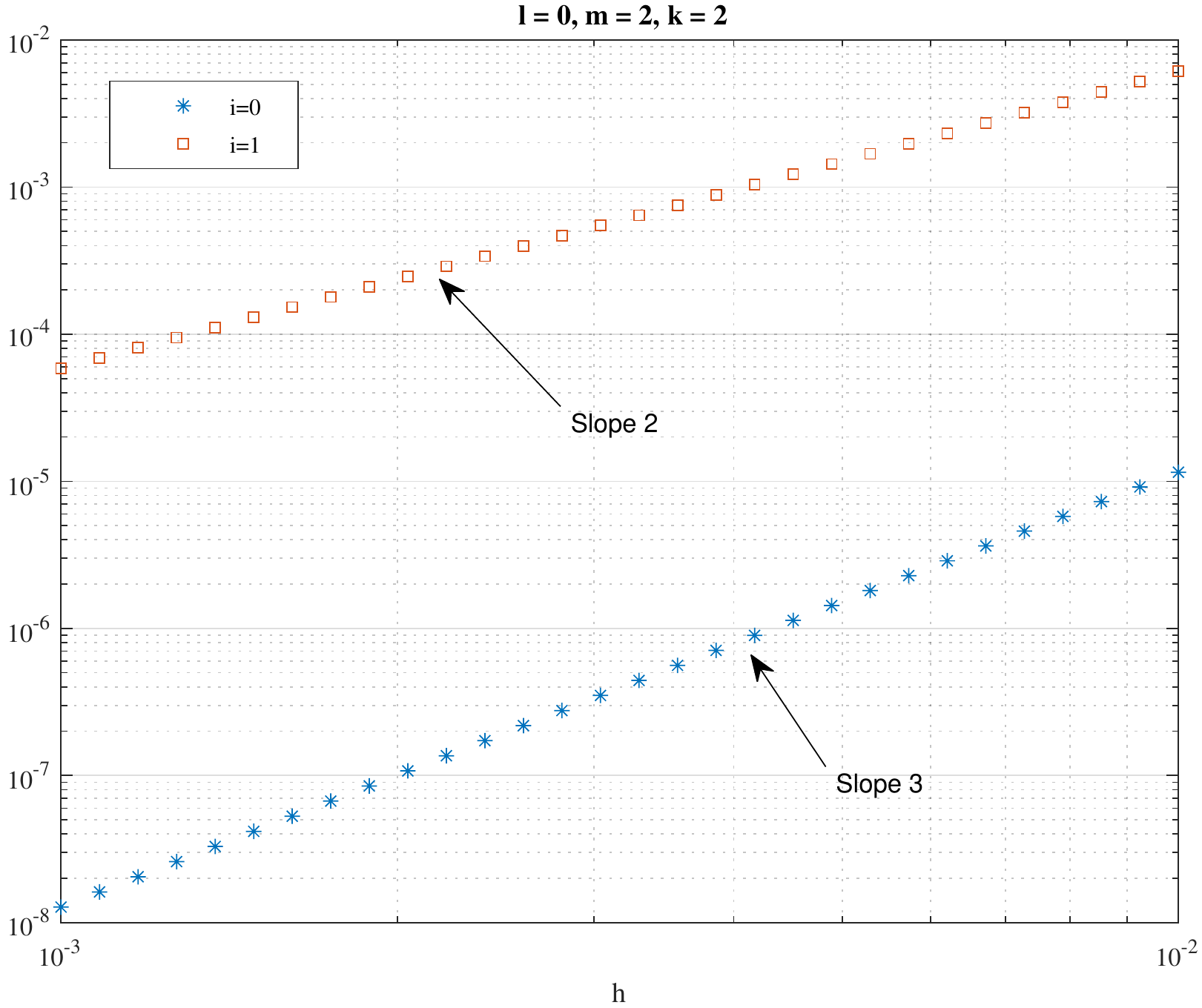}
\label{fig:img_pair_1}} & 
{\subfigure[$l=1,m=2,k=3$]{\includegraphics[scale=0.36]{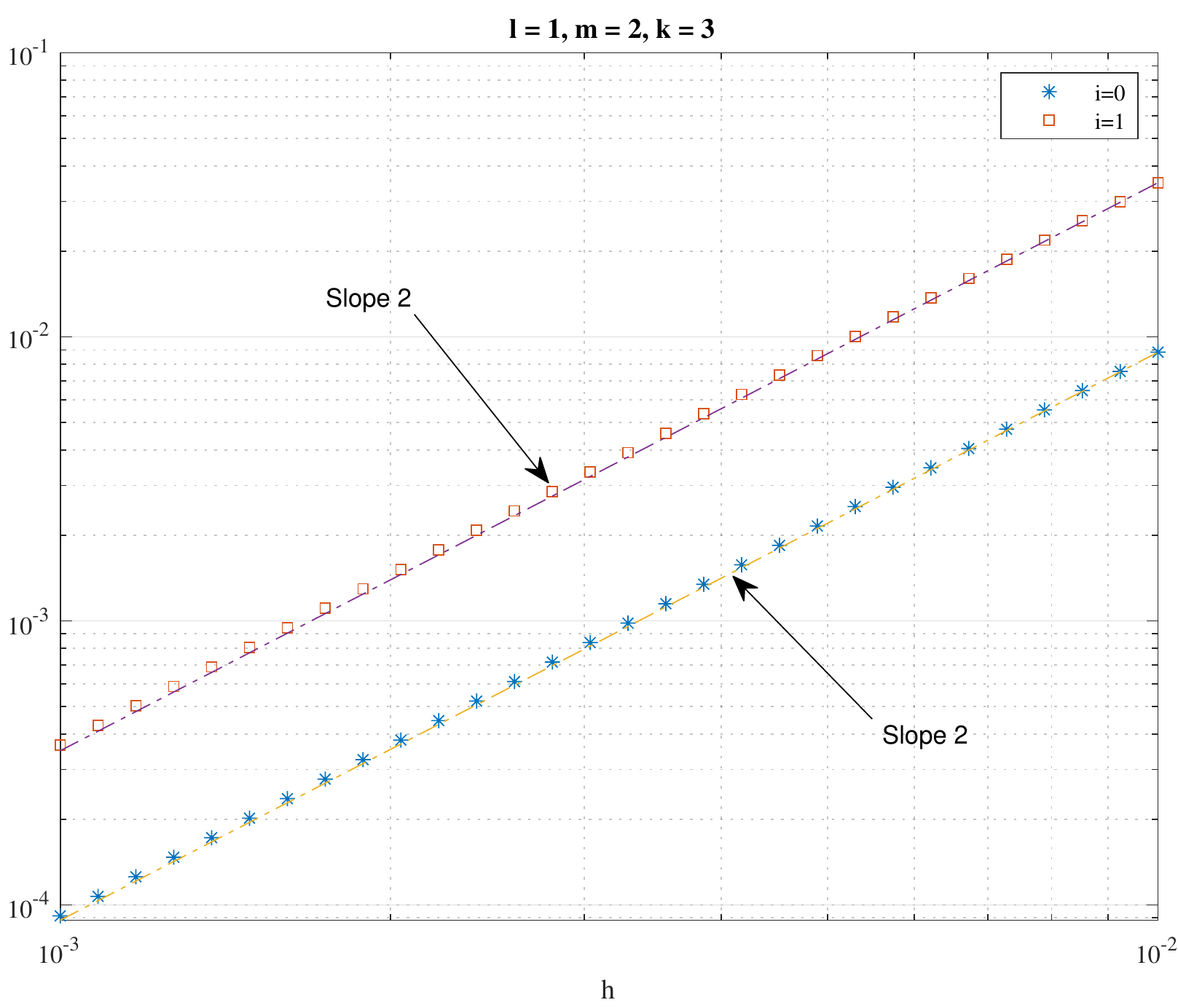}    
\label{fig:img_complete_1}}} \\
{\subfigure[$l=1,m=3,k=1$]{\includegraphics[scale=0.36]{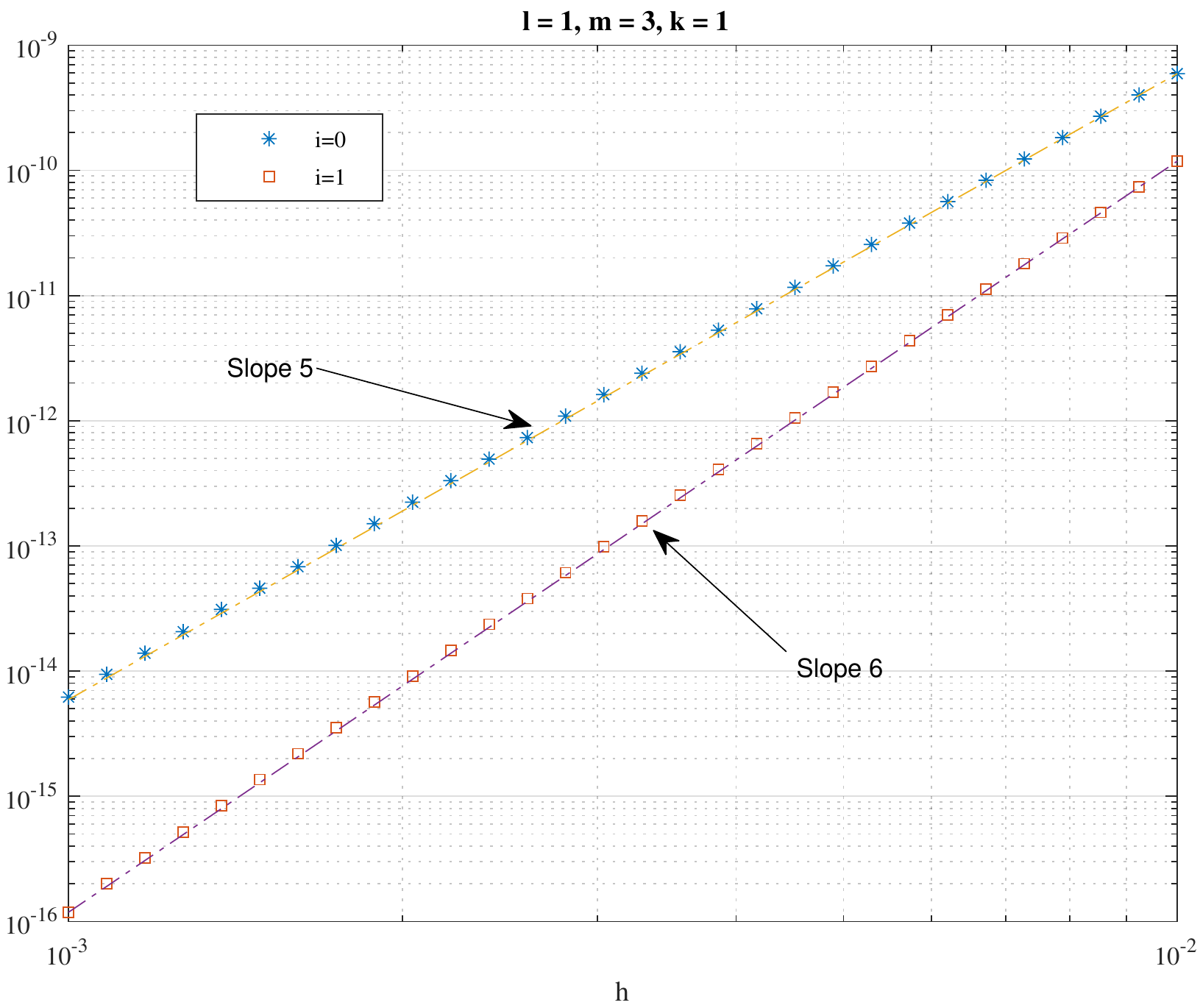}
\label{fig:img_complete_2}}} &
\subfigure[$l=1,m=3,k=3$]{\includegraphics[scale=0.36]{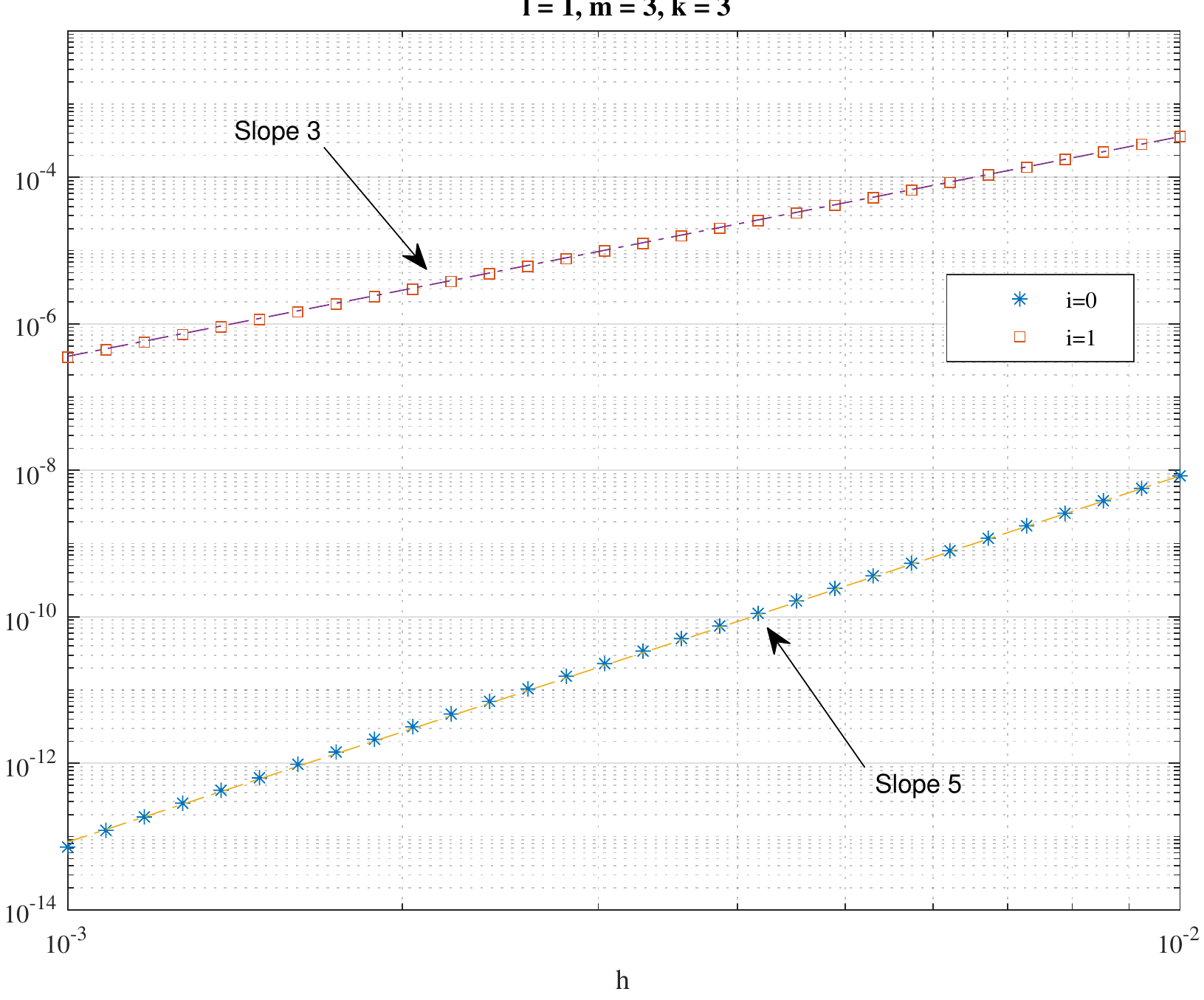}
\label{fig:img_pair_2}}\\
\end{tabular}
\vspace{10pt}
\caption[]{Log graph for the error $||\mathat{x}_1^{(i)}-\mat{x}_{\mathrm{ss}}^{(i)}(h)||$ at $i = 0,1$ versus the step size $h$. The highlighted slopes refer to the number of decades where the error drops within one decade of variation in $h$. The slopes match the orders predicted by theorem \ref{them:DAE-theorem-1} given the corresponding values of $l,m$ and $k$.}
\label{fig:affinity}
\vspace{-10pt}
\end{figure}

Worthy of observation on Figure \ref{fig:affinity} is behaviour of the error in Figures \ref{fig:img_complete_1} compared with that in \ref{fig:img_pair_2}. Those two panels show the method used for the same circuit (the one with $k=3$) but with different values for $m$. What needs to be noted here is the increase in the order from 2 in the former to 5 in the latter, which underscores the order reduction phenomena mentioned in Remark \ref{rem:main-result}.

\section{Conclusions}
\label{sec:conclusion}
This paper presented a novel theoretical result characterizing the local order of convergence of the recently proposed high-order $A$- and $L$-stable based on the Obreshkov formula. The main focus of this work has been the derivation of this order of convergence when the system of differential equations takes the form of differential-algebraic equations (DAE). The derived results showed that this order may be different from the order of convergence the in the ordinary differential equations (ODE) and for certain configurations of the Obreshkov formula. Theoretical results have been validated with careful numerical simulation of several circuits.

\end{document}